\documentclass{amsart}
\usepackage{amsthm, amssymb, amsfonts, amscd}
\usepackage{graphicx}
\usepackage{stmaryrd} 
\usepackage{tikz}
\usepackage{tikz-cd}
\usepackage{bbold}
\usetikzlibrary{arrows}
\usepackage{verbatim}
\usepackage[all,arc]{xy}
\usepackage{enumerate}
\usepackage{mathrsfs, mathabx}
\usepackage{xcolor}
\usepackage{hyperref}
\usepackage{soul}
\usepackage{comment}
\usepackage{esint} 
\usepackage[margin=1.0in]{geometry}
\usepackage[
open, openlevel=2, atend
]{bookmark}[2011/12/02]
\bookmarksetup{color=blue}

\theoremstyle{definition} 
\newtheorem{thm}{Theorem}[section]
\newtheorem{cor}[thm]{Corollary}
\newtheorem{prop}[thm]{Proposition}
\newtheorem{lem}[thm]{Lemma}
\newtheorem{conj}[thm]{Conjecture}

\theoremstyle{definition}
\newtheorem{defn}[thm]{Definition}

\newtheorem{rem}[thm]{Remark}

\newcommand{\mr}{\mathrm}
\newcommand{\mf}{\mathfrak}

\newcommand{\mc}{\mathcal}

\newcommand{\wt}{\widetilde}

\newcommand{\Aut}{\mr{Aut}}

\newcommand{\bb}{\mathbb}

\newcommand{\Z}{\bb{Z}}

\newcommand{\bP}{\bb{P}}

\newcommand{\bF}{\bb{F}}
\newcommand{\bE}{\bb{E}}
\newcommand{\fp}{\bF_p}
\newcommand{\fq}{\bF_q}
\newcommand{\zp}{\Z_p}

\newcommand{\Sur}{\mr{Sur}}

\newcommand{\Haar}{\mr{Haar}}

\newcommand{\M}{\mr{M}}

\newcommand{\sg}{\sigma}
\newcommand{\Sg}{\Sigma}

\newcommand{\cok}{\mr{cok}}

\newcommand{\cl}{\mr{cl}}
\newcommand{\CL}{\mr{CL}}

\newcommand{\ol}{\overline}

\newcommand{\lt}{\left}
\newcommand{\rt}{\right}

\newcommand{\be}{\begin{enumerate}}
\newcommand{\ee}{\end{enumerate}}

\newcommand{\bi}{\begin{itemize}}
\newcommand{\ei}{\end{itemize}}

\newcommand{\bbm}{\begin{bmatrix}}
\newcommand{\ebm}{\end{bmatrix}}

\newcommand{\lf}{\lfloor}
\newcommand{\rf}{\rfloor}

\numberwithin{equation}{section}
\allowdisplaybreaks

\title[Random $p$-adic matrices with prescribed zero patterns]{A mod $p$ determinant criterion for Cohen--Lenstra convergence of random $p$-adic matrices with prescribed zero patterns}

\author{Hyungmin Jang}
\address[H. Jang]{Department of Mathematics, Yonsei University, Seoul 03722, Republic of Korea}
\email{io1278@yonsei.ac.kr}

\author{Nathan Kaplan}
\address[N. Kaplan]{Department of Mathematics, University of California, Irvine, CA 92697, USA}
\email{nckaplan@math.uci.edu}

\author{Jungin Lee}
\address[J. Lee]{Department of Mathematics, Ajou University, Suwon 16499, Republic of Korea}
\email{jileemath@ajou.ac.kr}

\author{Myungjun Yu}
\address[M. Yu]{Department of Mathematics, Yonsei University, Seoul 03722, Republic of Korea}
\email{mjyu@yonsei.ac.kr}

\begin{document}

\begin{abstract}
We study the distribution of cokernels of Haar-random matrices over the $p$-adic integers with prescribed zero patterns, motivated by the Cohen--Lenstra heuristics.
A central feature of our approach is that the asymptotic cokernel distribution is governed by the reductions modulo $p$ of these matrices, viewed as random matrices over the finite field $\mathbb{F}_p$.
For several families of support patterns arising from stair-shaped zero regions, including general stair-shaped patterns, band matrices, and matrices with two symmetric stair-shaped zero regions, we show that convergence of the cokernel distribution to the Cohen--Lenstra distribution is equivalent to an asymptotic nonsingularity condition over $\mathbb{F}_p$.
We further propose a conjecture for general support patterns and give examples showing that analogous rank-$r$ criteria fail for $r\ge 1$.
\end{abstract}

\maketitle

\section{Introduction}

Let $p$ be a prime number and let $Y_\CL$ be the random finite abelian $p$-group whose probability distribution is given by
\begin{equation}
\label{eq: CL distribution}
\bP(Y_\CL \cong G) = \frac{1}{\#\Aut(G)}\prod_{i=1}^{\infty} (1-p^{-i}) =: \CL(G)
\end{equation}
for every finite abelian $p$-group $G$. The distribution of $Y_{\CL}$, referred to as the \emph{Cohen--Lenstra distribution}, arises naturally in many areas of mathematics. 

As the name suggests, this distribution first appeared in the work of Cohen and Lenstra \cite{CL84}, who proposed a conjecture concerning the distribution of the Sylow $p$-subgroups of the class groups of imaginary quadratic fields. 
Motivated by the function field analogue of the Cohen--Lenstra conjecture, Friedman and Washington \cite{FW89} studied the distribution of random matrices over the ring of $p$-adic integers $\zp$. They proved that if $X_{\Haar, n}$ is a random $n \times n$ matrix over $\zp$ whose entries are independent and distributed according to the Haar measure, then the distribution of the cokernel of $X_{\Haar, n}$ (denoted by $\cok(X_{\Haar, n})$) converges to the Cohen--Lenstra distribution, i.e.,
$$
\underset{n \to \infty}{\lim} \bP(\cok(X_{\Haar, n}) \cong G) = \CL(G)
$$
for every finite abelian $p$-group $G$. 

The Cohen--Lenstra distribution also appears in combinatorial contexts. Kahle and Newman \cite{KN22} conjectured that if $C_n$ is 
a random $2$-dimensional determinantal hypertree, the distribution of the Sylow $p$-subgroup of the first homology group $H_1(C_n)$ follows the Cohen--Lenstra distribution. 
We refer the reader to \cite{LY25, Mes23, Mes24a, Mes25} for recent progress in this direction.

By reducing \eqref{eq: CL distribution} modulo $p$, we have 
\begin{equation}
\label{eq: CL mod p distribution}
\bP\lt( Y_\CL/pY_{\CL} \cong \fp^k \rt) = p^{-k^2}\prod_{i=1}^{k} (1-p^{-i})^{-2}\prod_{i=1}^\infty (1-p^{-i}) =: \cl(k)
\end{equation}
for every nonnegative integer $k$ (see \cite[Theorem 6.3]{CL84}). Here, $\fp$ denotes the finite field with $p$ elements. In particular, we have
$$
\cl(m) = \sum_{G: ~\mr{rk}_p(G) = m} \CL(G),
$$
where the sum varies over all finite abelian $p$-groups $G$ whose $p$-rank $\mr{rk}_p(G) := \dim_{\fp} G/pG$ is equal to $m$.

As a generalization of the work of Friedman and Washington, Kang, Lee, and Yu \cite{KLY24} studied random $p$-adic matrices with fixed zero entries. In particular, they investigated when the distribution of the cokernel of an $n \times n$ matrix over $\zp$ converges to the Cohen--Lenstra distribution as $n \to \infty$, in a setting where some entries are fixed to be $0$ and the remaining entries are independent and Haar-random in $\zp$. In \cite{Woo19}, Wood proved strong universality theorems for cokernels of random $n \times n$ matrices over $\zp$.  These theorems require that each entry of the random matrix is not too concentrated on any value modulo $p$, which excludes matrices where some entries are fixed to be $0$.  This provides motivation to study cokernels of random matrices where certain entries are fixed, seeking to establish universality phenomena for a broader class of random matrix models.  Further interest in this direction comes from applications to the study of random graphs, building on Wood's breakthrough results about sandpile groups of Erdős--Rényi random graphs \cite{Woo17}.  As discussed in \cite[Section 1.2]{KLY24}, if one considers a random graph on $n$ vertices in which some edges can never exist and all other edges occur independently with probability $q \in (0,1)$, then its sandpile group can be represented as the cokernel of a random symmetric matrix whose prescribed zero entries reflect the forbidden edges.

Kang, Lee, and Yu established a general lower bound on the number of random entries (those not fixed to be $0$) required for such convergence \cite[Theorem 4.1]{KLY24}. Furthermore, when the fixed zero entries form a stair-shaped pattern, they provided a sufficient condition ensuring convergence to the Cohen--Lenstra distribution \cite[Theorems 3.3 and 3.4]{KLY24}. In this paper, we prove that these conditions are not only sufficient but also necessary (see Corollaries \ref{cor: stair d=1} and \ref{cor: stair d>1}).

\begin{defn}
Let $(Y_n)_{n=1}^{\infty}$ be a sequence of random finitely generated $\zp$-modules. We say that \emph{$Y_n$ converges to CL} (denoted by $Y_n \rightrightarrows \CL$) if the distribution of $Y_n$ converges to the Cohen--Lenstra distribution as $n \to \infty$, i.e., for every finite abelian $p$-group $G$,
$$
\lim_{n\to\infty}\bP(Y_n \cong G) = \CL(G).
$$
Let $m$ be a nonnegative integer. We say that \emph{$Y_n$ converges to CL in rank $m$} if 
$$
\underset{n \to \infty}{\lim}\bP(Y_n/pY_n \cong \fp^m) = \cl(m). 
$$
\end{defn}

It is clear that if $Y_n$ converges to CL, then $Y_n$ converges to CL in rank $m$ for all nonnegative integers $m$. However, the converse does not hold. For example, let $Y = Y_1 = Y_2 = \cdots$ be a random finite abelian $p$-group given by 
$$
\bP(Y \cong (\Z/p\Z)^m) = \cl(m)
$$
for every $m \ge 0$, and
$$
\bP(Y \cong H) = 0
$$
for all non-elementary $p$-groups $H$ (i.e., $pH \ne 0$). It is clear that $Y_n$ does not converge to CL, whereas $Y_n$ converges to CL in rank $m$ for all $m\ge 0$.

In the case of interest---namely, the distribution of the cokernel of a Haar-random matrix with some entries fixed to $0$---we observe that the converse often holds. Even more remarkably, assuming convergence to CL in rank $0$ already suffices to deduce convergence to CL in many cases.
As an example, we prove the following. See Section \ref{Sub21} for notation and terminology.

\begin{thm}[Theorem \ref{thm: stair main theorem}]
\label{thm: intro thm 1}
Assume that for every positive integer $n$,
$$\sg_{n,1} \subseteq \sg_{n,2}\subseteq \cdots \subseteq \sg_{n,n} \subseteq [n].
$$ 
Let $X_n \in \M_n(\zp)$ be the Haar-random matrix over $\zp$ supported on $\Sg_n = (\sg_{n,1}, \sg_{n,2}, \ldots, \sg_{n,n})$. Then 
$$
\cok(X_n) \rightrightarrows \CL~\Longleftrightarrow~ \cok(X_n) ~\text{converges to}~\CL~\text{in rank $0$}. 
$$
\end{thm}

For a matrix $A_n \in \M_n(\zp)$, let $\ol{A_n} \in \M_n(\fp)$ denote the reduction of $A_n$ modulo $p$. Then
$$
\cok(A_n) = 0 ~\Longleftrightarrow ~\cok(\ol{A_n}) = 0 ~\Longleftrightarrow ~ \det(\ol{A_n}) \ne 0,
$$
where the first equivalence follows from the fact that
$$
\cok(\ol{A_n}) \cong \cok(A_n) \otimes \fp. 
$$
This implies that the condition on the right side of Theorem \ref{thm: intro thm 1} is equivalent to the requirement that
$$
\underset{n \to \infty}{\lim}\bP(\det(\ol{X_n}) \ne 0) = \cl(0) = \prod_{i=1}^\infty (1-p^{-i}).
$$
Therefore, Theorem \ref{thm: intro thm 1} reduces the problem of proving convergence to CL to a much weaker condition on the limiting probability that the reduction of $X_n$ modulo $p$ is nonsingular.

We further show that the same conclusion holds under various patterns of random entries of $X_n$. 
First, we consider a band matrix. More precisely, let $(t_n)_{n=1}^\infty$ be a sequence of positive integers such that $t_n \le n$ for all $n$, and $\sg_{n,i} = \{j \in [n] : |j-i| \le t_n\}$. The Haar-random matrix $X_n \in \M_n(\zp)$ supported on $\Sg_n = (\sg_{n,1}, \ldots, \sg_{n,n})$ is called the \emph{Haar-random band matrix of width $t_n$}. 
Mészáros investigated the precise conditions under which $\cok(X_n)$ converges to CL, and obtained the following result \cite[Theorem 1]{Mes24b}. Related questions for band matrices have also been studied over the real numbers, especially in connection with localization and delocalization for Gaussian band matrices; see \cite[Section 1.2]{Mes24b}.

\begin{thm}[Mészáros]
\label{thm: Meszaros thm}
Let $X_n \in \M_n(\zp)$ be the Haar-random band matrix of width $t_n$. Then the following are equivalent. 
\begin{enumerate}
    \item $\cok(X_n) \rightrightarrows \CL$,
    \item $\underset{n \to \infty}{\lim} (t_n - \log_p n) = \infty$.
\end{enumerate}
\end{thm}

Using Theorem \ref{thm: Meszaros thm}, we prove the following theorem.

\begin{thm}[Theorem \ref{thm: band main theorem}] \label{thm: intro thm 2}
Let $X_n \in \M_n(\zp)$ be the Haar-random band matrix of width $t_n$. Then
$$
\cok(X_n) \rightrightarrows \CL ~\Longleftrightarrow~ \underset{n \to \infty}{\lim}\bP(\det(\ol{X_n}) \ne 0) = \cl(0).
$$
\end{thm}

Note that a band matrix can be regarded as a matrix whose zero entries form two symmetric stair-shaped regions, each step having height $1$ and width $1$. More generally, we consider Haar-random matrices over $\zp$ with two symmetric stair-shaped zero regions, where each step has height $d \ge 2$ and width $1$. 

\begin{thm}[Theorem \ref{thm: two stairs, height d width 1}] \label{thm: intro thm 3}
Let $d \ge 2$, $(t_n)_{n=1}^\infty$ be a sequence of nonnegative integers such that $dt_n \le n$ for all $n$ and $X_n \in \M_n(\zp)$ be the Haar-random matrix supported on $\Sg_n = (\sg_{n,1}, \ldots, \sg_{n,n})$, where 
\begin{equation*}
\sg_{n, i} = \begin{cases}
[n-d(t_n-i+1)] & ~\text{if $1 \le i \le t_n$} \\
[n] & ~\text{if $t_n + 1 \le i \le n-t_n$} \\
[n] \setminus [d(i-n+t_n)] & ~\text{if $n- t_n + 1 \le i \le n$}.
\end{cases}
\end{equation*}
Then 
$$
\cok(X_n) \rightrightarrows \CL ~\Longleftrightarrow~ \underset{n \to \infty}{\lim}\bP(\det(\ol{X_n}) \ne 0) = \cl(0).
$$
\end{thm}

Theorems \ref{thm: intro thm 1}, \ref{thm: intro thm 2} and \ref{thm: intro thm 3} suggest the following conjecture.
\begin{conj}
\label{conj: main conjecture}
Let $X_n \in \M_n(\zp)$ be the Haar-random matrix over $\zp$ supported on $\Sg_n$. Then
$$
\cok(X_n) \rightrightarrows \CL ~\Longleftrightarrow~ \underset{n \to \infty}{\lim}\bP(\det(\ol{X_n}) \ne 0) = \cl(0). 
$$
\end{conj}

Our proofs rely primarily on the strategy developed by Wood \cite{Woo17, Woo19} to determine the limiting distributions of cokernels of various families of random matrices over $\zp$. 
More precisely, we use the following theorem, which states that the convergence of (surjective) moments implies convergence in distribution.
When $G$ and $H$ are finitely generated $\zp$-modules, we denote by $\Sur(G, H)$ the set of all surjective homomorphisms from $G$ to $H$. 

\begin{thm}[Wood \cite{Woo19}]
\label{thm: Wood robust uniqueness theorem}
Let $\mc{P}$ be the set of all finite abelian $p$-groups and $Y_1, Y_2, \ldots$ be a sequence of random finitely generated $\zp$-modules. Suppose that
$$
\underset{n \to \infty}{\lim}\bE(\#\Sur(Y_n, H)) = 1
$$
for all $H \in \mc{P}$. Then $Y_n$ converges to CL, i.e., for every $G \in \mc{P}$,
$$
\underset{n \to \infty}{\lim} \bP(Y_n \cong G) = \CL(G). 
$$
\end{thm}

To prove our main theorems, we translate the condition $\underset{n \to \infty}{\lim}\bP(\det(\ol{X_n}) \ne 0) = \cl(0)$ (equivalently, $\cok(X_n)$ converges to CL in rank $0$) into a formula that enables us to conclude that the moments converge to the desired value $1$. For general $\Sg_n$, however, carrying out this translation appears to be a highly nontrivial task. This is precisely why our main theorems are established only under certain restrictions on $\Sg_n$.  

The paper is organized as follows. In Section \ref{Sec2}, we set up the notation and prove Theorems \ref{thm: intro thm 1}, \ref{thm: intro thm 2}, and \ref{thm: intro thm 3}, which are special cases of Conjecture \ref{conj: main conjecture}.
In Section \ref{Sec3}, we show that Conjecture \ref{conj: main conjecture} does not hold if the condition of convergence to CL in rank $0$ is replaced with convergence to CL in rank $r$ for some positive integer $r$. 
In Section \ref{Sec4}, we prove that the probability that $\ol{X_n}$ is nonsingular increases as the support $\Sg_n$ grows, which can be viewed as evidence for Conjecture \ref{conj: main conjecture}.

\section{From rank-\texorpdfstring{$0$}{0} convergence to CL convergence} \label{Sec2}

\subsection{Notation and terminology} \label{Sub21}

The following notation will be used throughout the paper.

\begin{itemize}
    \item Let $p$ be a fixed prime, $\fp$ be the finite field with $p$ elements and $\zp$ be the ring of $p$-adic integers. For every positive integer $n$, let $[n] = \{1, 2, \ldots, n\}$.

    \item Let $c_p(m) = \prod_{i=1}^{m} (1-p^{-i})$ and $c_p = \cl(0)= \prod_{i=1}^{\infty} (1-p^{-i})$. Note that $\cl(m) = \frac{c_p}{p^{m^2}c_p(m)^2}$.

    \item For a commutative ring $R$, let $\M_n(R)$ (resp. $\M_{n\times(n+u)}(R)$) denote the set of all $n \times n$ matrices (resp. $n \times (n+u)$ matrices) over $R$. For $A \in \M_n(R)$ and $i,j \in [n]$, let $A_{i, j}$ be the $(i, j)$-th entry of $A$. For $A \in \M_n(\zp)$, let $\ol{A} \in \M_n(\fp)$ be the reduction of $A$ modulo $p$. For $A \in \M_{n\times (n+u)}(\zp)$, define $\ol{A} \in \M_{n\times (n+u)}(\fp)$ similarly.

    \item A random variable $\xi$ taking values in $\zp$ is called \emph{Haar-random} if its distribution is given by the Haar probability measure on $\zp$; equivalently, $\bP(\xi \equiv a \pmod{p^n}) = \frac{1}{p^n}$ for all positive integers $n$ and $a \in \Z/p^n\Z$. 

    \item Let $\sg_{n,1}, \ldots, \sg_{n,n}$ be subsets of $[n]$ and $\Sg_n := (\sg_{n,1}, \ldots, \sg_{n,n})$. Let $X_n \in \M_n(\zp)$ be a random $n \times n$ matrix such that the entries $(X_n)_{i,j}$ with $i \in \sg_{n,j}$ are Haar-random and independent, and if $i \notin \sg_{n,j}$ then $(X_n)_{i,j}=0$. 
    In this case, we say $X_n$ is the \emph{Haar-random matrix supported on} $\Sg_n$. 

    \item Let $V = \zp^n$ and $v_1, v_2, \ldots, v_n$ denote the standard basis for $V$. For $\tau \subseteq [n]$, we write $V_\tau$ for the submodule of $V$ spanned by the $v_i$ with $i \in \tau$.

    \item The letters $G$ and $H$ will always denote finite abelian $p$-groups. 
\end{itemize}

For a positive integer $n$, fix $\Sg_n = (\sg_{n,1}, \ldots, \sg_{n,n})$ and let $X_n \in \M_n(\zp)$ be the Haar-random matrix supported on $\Sg_n$. 
Since every surjection $\cok(X_n) \to G$ lifts uniquely to a surjection $F: \zp^n \to G$ satisfying $FX_n=0$, we have
$$
\bE(\# \Sur(\cok(X_n), G))
= \sum_{F \in \Sur(\zp^n, G)} \bP(FX_n = 0).
$$
Since the entries of $X_n$ are Haar-random and independent, we have 
$$
\bP(FX_n = 0) = \prod_{i=1}^n \bP(F(X_n)_{*,i} = 0) =\frac{1}{|FV_{\sg_{n,1}}| \cdots |FV_{\sg_{n,n}}|}
$$
for every $F \in \Sur(\zp^n, G)$, where $(X_n)_{*,i}$ denotes the $i$-th column of $X_n$. 
Thus, it follows that
$$
\bE(\# \Sur(\cok(X_n), G)) = \sum_{F \in \Sur(\zp^n, G)} \frac{1}{| FV_{\sg_{n,1}}| \cdots| FV_{\sg_{n,n}}|}.
$$
For $G_1, \ldots, G_n \le G$, write
$$
d_{G_1, \ldots, G_n} :=  \frac{ \#\lt \{ F \in \Sur(\zp^n, G) \mid FV_{\sg_{n,i}} = G_i \text{ for } 1 \le i \le n \rt \}}{|G_1| \cdots |G_n|}.
$$

The following lemma provides a useful criterion to show that moments are converging to $1$. This allows us to apply Theorem \ref{thm: Wood robust uniqueness theorem} in various cases.

\begin{lem} 
\label{lem: d moment error term}
(\cite[Proposition 3.1]{KLY24}) We have
$$
\underset{n \to \infty}{\lim} \bE(\# \Sur(\cok(X_n), G)) = 1 ~\Longleftrightarrow~ \underset{n \to \infty}{\lim} \sum_{\substack{(G_1, \ldots, G_n) \\ \ne (G, \ldots ,G)}}  d_{G_1, \ldots, G_n} = 0. 
$$
\end{lem}

\subsection{Stair-shaped zero regions} \label{Sub22}

Suppose that for each positive integer $n$ there exists a permutation $\tau_n: [n] \to [n]$ such that 
$$
\sg_{n,\tau_n(1)} \subseteq \sg_{n,\tau_n(2)} \subseteq \cdots \subseteq \sg_{n,\tau_n(n)}.
$$
For $A \in \M_n(\zp)$, the cokernel $\cok(A)$ is invariant under elementary row and column operations up to isomorphism. 
Therefore, after permuting rows and columns, it suffices to consider the following \emph{stair-shaped zero region} case.
Let $0 \le \ell_n < n$ be an integer. Let $1 \le h(n)_1 < \cdots < h(n)_{\ell_n} < n$ and $1\le v(n)_1 < v(n)_2< \cdots < v(n)_{\ell_n} < n$ be strictly increasing sequences of positive integers and $h(n)_0 = v(n)_0=0$. Define $\Sg_n=(\sg_{n,1}, \ldots, \sg_{n,n})$ by 
\begin{equation}\label{eq: stair sigma}
\sg_{n, i} = \begin{cases}
[v(n)_a] & ~\text{if $h(n)_{a-1} + 1\le i \le h(n)_a$ for some $1\le a \le \ell_n$} \\
[n] & ~\text{if $h(n)_{\ell_n} + 1 \le i \le n$}.
\end{cases}
\end{equation}

\begin{figure}[ht]
\begin{equation*}
\begin{pmatrix}
* & * & * & * & * & *\\ 
* & * & * & * & * & *\\ 
0 & * & * & * & * & *\\ 
0 & 0 & 0 & * & * & *\\ 
0 & 0 & 0 & 0 & 0 & *\\
0 & 0 & 0 & 0 & 0 & *
\end{pmatrix}
\end{equation*}
\caption{A matrix $X_n \in \M_n(\zp)$ for $(n, \ell_n)=(6,3)$, $(h(n)_1, h(n)_2, h(n)_3)=(1,3,5)$ and $(v(n)_1, v(n)_2, v(n)_3)=(2,3,4)$}.
\label{fig1}
\end{figure}

By the definition of $\sg_{n,i}$, we see that $\ell_n$ is the number of steps in the stair-shaped zero region.

\begin{lem}
\label{lem: vertical geq horizontal}
If $\bP(\det(\ol{X_n}) \ne 0) >0$, then $v(n)_{i} \ge h(n)_i$ for all $1\le i \le \ell_n$.
\end{lem}

\begin{proof}
If $v(n)_i < h(n)_i$, then the first $h(n)_i$ columns of $\ol{X_n}$ are always linearly dependent.
\end{proof}

\begin{thm}
\label{thm: stair main theorem}
Let $X_n \in \M_n(\zp)$ be the Haar-random matrix supported on $\Sg_n = (\sg_{n,1}, \ldots, \sg_{n,n})$, where $\sg_{n,i}$ are given as in \eqref{eq: stair sigma}. Then the following are equivalent:
\begin{enumerate}
\item $\underset{n \to \infty}{\lim} \bP(\det(\ol{X_n}) \ne 0) = c_p$,
\item $\underset{n \to \infty}{\lim}\sum_{i=1}^{\ell_n} \frac{1}{p^{v(n)_i - h(n)_i}} = 0$,
\item $\cok(X_n) \rightrightarrows \CL$.
\end{enumerate}
\end{thm}

\begin{proof}
We first show (1) $\Longrightarrow$ (2). By Lemma \ref{lem: vertical geq horizontal}, we may take $n$ large enough so that $v(n)_i \ge h(n)_i$ for all $1 \le i \le \ell_n$. The matrix $\overline{X_n}$ is nonsingular if and only if, for all $1 \le j \le n$, the $j$-th column of $\overline{X_n}$ is not contained in the subspace 
generated by the first $j-1$ columns of $\overline{X_n}$. It follows that
\begin{equation}
\label{eq: determinant stair}
\bP(\det(\ol{X_n}) \ne 0) =  \lt(\prod_{i=1}^{\ell_n} \prod_{k=h(n)_{i-1}}^{h(n)_i-1} \lt(1- \frac{p^k}{p^{v(n)_i}}\rt) \rt) \lt(\prod_{k=h(n)_{\ell_n}}^{n-1} \lt(1-\frac{p^k}{p^n}\rt) \rt). 
\end{equation}
As $v(n)_{i} \le n-1$ for each $i$, we have
$$
\bP(\det(\ol{X_n}) \ne 0) \le \lt(\prod_{k=0}^{h(n)_{\ell_n}-1} \lt(1-\frac{p^k}{p^{n-1}}\rt)\rt)
\lt(\prod_{k=h(n)_{\ell_n}}^{n-1} \lt(1-\frac{p^k}{p^n}\rt) \rt)
= (1-p^{-n})^{-1} \lt(1- \frac{p^{h(n)_{\ell_n}}}{p^{n}}\rt)c_p(n). 
$$
Taking the limit as $n \to \infty$, it follows from condition (1) that
$$
\underset{n \to \infty}{\lim} (n-h(n)_{\ell_n}) = \infty. 
$$
Taking the limit as $n \to \infty$ in \eqref{eq: determinant stair}, we have
$$
\underset{n \to \infty}{\lim}\prod_{i=1}^{\ell_n} \prod_{k=h(n)_{i-1}}^{h(n)_i-1} \lt(1- \frac{p^k}{p^{v(n)_i}} \rt) = 1.
$$
Now condition (2) follows from the inequality
$$
\prod_{i=1}^{\ell_n} \prod_{k=h(n)_{i-1}}^{h(n)_i-1} \lt(1- \frac{p^k}{p^{v(n)_i}} \rt) \le \prod_{i=1}^{\ell_n} \lt(1-\frac{1}{p^{v(n)_i - (h(n)_i-1)}} \rt)
\le \exp \lt(- \sum_{i=1}^{\ell_n} \frac{1}{p^{v(n)_i - (h(n)_i-1)}} \rt).
$$
Now we prove (2) $\Longrightarrow$ (3). By condition (2), we may assume that $v(n)_i \ge h(n)_i$. Let $r$ be a positive integer. 
Taking the $r$-th power of the sum $\sum_{i=1}^{\ell_n} \frac{1}{p^{v(n)_i - h(n)_i}}$, we obtain
\begin{equation}
\label{eq: reciprocal consequence general}
\underset{n \to \infty}{\lim}\sum_{0 = i_0 <i_1 < \cdots < i_r \le \ell_n} \frac{1}{p^{(v(n)_{i_r} + \cdots + v(n)_{i_1}) - (h(n)_{i_r} + \cdots + h(n)_{i_1})}} = 0.
\end{equation}
Let $G$ be a finite abelian $p$-group with $|G| = p^m$. 
By Lemma \ref{lem: d moment error term}, it is enough to prove that 
$$
\underset{n \to \infty}{\lim} \sum_{\substack{G_1 \le \cdots \le G_{\ell_n} \\ G_1 \ne G}}  d_{G_1^{h(n)_{1}},~G_2^{h(n)_{2}-h(n)_{1}},~\ldots~,~ G_{\ell_n}^{h(n)_{\ell_n} - h(n)_{\ell_{n}-1}},~ G^{n- h(n)_{\ell_n}}} = 0.
$$
(Here, the superscript $G_i^{t}$ indicates that the subgroup $G_i$ is repeated $t$ times).
Consider the set 
$$
\mr{CS}_G := \lt \{ (H_1, \ldots, H_{r+1}) \mid 1 \le r \le m \text{ and } H_1 \lneq H_2 \lneq \cdots \lneq H_{r+1} = G \rt \}. 
$$

Let  $(H_1, \ldots, H_{r+1}) \in \mr{CS}_G$ and $\frac{|G|}{|H_{a}|} = p^{m_a}$ for $1\le a \le r+1$, so $0 = m_{r+1} < m_r < \cdots < m_1$. For every $0 = i_0 < i_1 < \cdots < i_r \le \ell_n$, it follows that 
\begin{align*}
\prod_{a=1}^r \lt(\frac{|H_a|}{|G|}\rt)^{(v(n)_{i_a} - v(n)_{i_{a-1}} )- (h(n)_{i_a}- h(n)_{i_{a-1}})} & = 
\prod_{a=1}^r \lt(\frac{1}{p^{m_a}}\rt)^{(v(n)_{i_a} - h(n)_{i_a})- (v(n)_{i_{a-1}}- h(n)_{i_{a-1}})} \\
& = \frac{1}{p^{\sum_{a=1}^r (m_a-m_{a+1})(v(n)_{i_a}-h(n)_{i_a})}} \\
& \le \frac{1}{p^{\sum_{a=1}^r (v(n)_{i_a}-h(n)_{i_a})}}. 
\end{align*}
Then we have
\begin{align*}
& \sum_{\substack{G_1 \le \cdots \le G_{\ell_n} \\ G_1 \ne G}}  d_{G_1^{h(n)_{1}},~G_2^{h(n)_{2}-h(n)_{1}},~\ldots~,~ G_{\ell_n}^{h(n)_{\ell_n} - h(n)_{\ell_{n}-1}},~ G^{n- h(n)_{\ell_n}}}  \\
= & \sum_{(H_1, \ldots, H_{r+1}) \in \mr{CS}_G} \sum_{0 = i_0 < \cdots < i_{r} \le \ell_n} \frac{ | \lt \{ F \in \Sur(\zp^n, G) \mid FV_{\sg_{n,k}} = H_a \text{ if } h(n)_{i_{a-1}} < k \le h(n)_{i_a} \rt \} |}{ | H_1 |^{h(n)_{i_1}} | H_2 |^{h(n)_{i_2}-h(n)_{i_1}} \cdots | H_{r} |^{h(n)_{i_{r}}-h(n)_{i_{r-1}}} |G|^{n-h(n)_{i_r}}} \\
\le & \sum_{(H_1, \ldots, H_{r+1}) \in \mr{CS}_G} \sum_{0 = i_0 < \cdots < i_{r} \le \ell_n} 
\frac{ | H_1 |^{v(n)_{i_1}} | H_2 |^{v(n)_{i_2}-v(n)_{i_1}} \cdots | H_{r} |^{v(n)_{i_r}-v(n)_{i_{r-1}}} |G|^{n- v(n)_{i_r}}}{ | H_1 |^{h(n)_{i_1}} | H_2 |^{h(n)_{i_2}-h(n)_{i_1}} \cdots | H_{r} |^{h(n)_{i_{r}}-h(n)_{i_{r-1}}} |G|^{n-h(n)_{i_r}}} \\
= & \sum_{(H_1, \ldots, H_{r+1}) \in \mr{CS}_G} \sum_{0 = i_0 < \cdots < i_{r} \le \ell_n} \prod_{a=1}^{r}  \lt(\frac{|H_a|}{|G|}\rt)^{(v(n)_{i_a} - v(n)_{i_{a-1}} )- (h(n)_{i_a}- h(n)_{i_{a-1}})} \\
\le & | \mr{CS}_G | \sum_{r=1}^{m} \sum_{0 = i_0 < \cdots < i_{r} \le \ell_n}\frac{1}{p^{(v(n)_{i_r} + \cdots + v(n)_{i_1}) - (h(n)_{i_r} + \cdots + h(n)_{i_1})}}
\end{align*}
which converges to $0$ by \eqref{eq: reciprocal consequence general}. This proves (2) $\Longrightarrow$ (3). The implication (3) $\Longrightarrow$ (1) is trivial.
\end{proof}

As a special case of Theorem \ref{thm: stair main theorem}, we consider the stair-shaped zero regions where each step has height $1$ and width $d$. The following corollaries show that the converses of \cite[Theorems 3.3 and 3.4]{KLY24} hold. 

\begin{cor}
\label{cor: stair d=1}
Let $(t_n)_{n=1}^\infty$ be a sequence of positive integers such that $t_n \le n$ for each $n$. Let
$$
\sg_{n,i} = \left\{\begin{matrix}
[t_n+(i-1)] & (1 \le i \le n-t_n)\\ 
[n] & (i \ge n-t_n+1)
\end{matrix}\right.
$$
for each $n$ and $i$. Then the following are equivalent:
\begin{enumerate}
    \item $\underset{n \to \infty}{\lim} \bP(\det(\ol{X_n}) \ne 0) = c_p$,
    \item $\underset{n \to \infty}{\lim} (t_n - \log_p n) = \infty$,
    \item $\cok(X_n) \rightrightarrows \CL$.
\end{enumerate}
\end{cor}

\begin{proof}
The subsets $\sg_{n,i}$ correspond to the case where $\ell_n = n-t_n$, $h(n)_i = i$ and $v(n)_i = t_n+i-1$. Since
$$
\underset{n \to \infty}{\lim} \sum_{i=1}^{\ell_n} \frac{1}{p^{v(n)_i - h(n)_i}} = \underset{n \to \infty}{\lim} \frac{n-t_n}{p^{t_n-1}} = 0
$$
if and only if $\underset{n \to \infty}{\lim} (t_n - \log_p n) = \infty$, the conclusion follows from Theorem \ref{thm: stair main theorem}. 
\end{proof}

\begin{cor}
\label{cor: stair d>1}
Let $d \ge 2$ be an integer. Let $(t_n)_{n=1}^\infty$ be a sequence of positive integers such that $t_n \le n$ for each $n$. Let
$$
\sg_{n,i} = \left\{\begin{matrix}
[t_n+(\lceil \frac{i}{d} \rceil -1)] & (1 \le i \le \min(d(n-t_n), n))\\ 
[n] & (i \ge d(n-t_n)+1)
\end{matrix}\right.
$$
for each $n$ and $i$. Then the following are equivalent:
\begin{enumerate}
    \item $\underset{n \to \infty}{\lim} \bP(\det(\ol{X_n}) \ne 0) = c_p$,
    \item $\underset{n \to \infty}{\lim} (n- d(n-t_n)) = \infty$,
    \item $\cok(X_n) \rightrightarrows \CL$.
\end{enumerate}
\end{cor}

\begin{proof}
The subsets $\sg_{n,i}$ correspond to the case where $\ell_n = n-t_n$, $h(n)_i = di$ and $v(n)_i = t_n+i-1$. Since 
$$
\underset{n \to \infty}{\lim} \sum_{i=1}^{\ell_n} \frac{1}{p^{v(n)_i - h(n)_i}} = \underset{n \to \infty}{\lim} \sum_{i=1}^{n-t_n} \frac{1}{p^{t_n-(d-1)i-1}} = 0
$$
if and only if $\underset{n \to \infty}{\lim} (t_n-(d-1)(n-t_n)-1) = \infty$, the conclusion follows from Theorem \ref{thm: stair main theorem}. 
\end{proof}

We also point out that Corollary~\ref{cor: stair d>1}, or more precisely its transposed version, will be applied below in the proof of Theorem~\ref{thm: two stairs, height d width 1}.

\subsection{Band matrices} \label{Sub23}

We next consider band matrices. Let $(t_n)_{n=1}^\infty$ be a sequence of positive integers such that $t_n \le n$ for all $n$. For each $n$ and $i$, let $\sg_{n,i} = \{j \in [n] : |j-i| \le t_n\}$. The Haar-random matrix $X_n \in \M_n(\zp)$ supported on $\Sg_n = (\sg_{n,1}, \ldots, \sg_{n,n})$ is called the \emph{Haar-random band matrix of width $t_n$}. 
Using Theorem \ref{thm: Meszaros thm}, we prove that Conjecture \ref{conj: main conjecture} holds for Haar-random band matrices. 

\begin{thm}
\label{thm: band main theorem}
Let $X_n \in \M_n(\zp)$ be the Haar-random band matrix of width $t_n$. Then 
$$
\cok(X_n) \rightrightarrows \CL ~\Longleftrightarrow~ \underset{n \to \infty}{\lim}\bP(\det(\ol{X_n}) \ne 0) = c_p.
$$
\end{thm}

\begin{proof}
By Theorem \ref{thm: Meszaros thm}, it suffices to show that 
$$\underset{n \to \infty}{\lim} \bP(\det(\ol{X_n}) \ne 0) = c_p \quad \Longrightarrow \quad \underset{n \to \infty}{\lim} (t_n - \log_p n) = \infty.$$

Let $e_1, \ldots, e_n$ denote the standard basis of $\fp^n$. Let $x_i$ be the $i$-th column of $\ol{X_n}$ (viewed as a random element in $\fp^n$). 
For $1\le i \le n$, let $c(i) = \max(1, i-t_n)$ and $d(i) = \min(n, i+t_n)$. 
Define
\begin{align*}
U_i & := \langle e_1, e_2, \ldots, e_{d(i)} \rangle, \\
V_i & := \langle x_1, x_2, \ldots, x_{i-1} \rangle, \\
W_i & := \langle e_{c(i)}, e_{c(i)+1}, \ldots, e_{d(i)}\rangle.
\end{align*}
The natural inclusion $W_i/(W_i \cap V_i) \hookrightarrow U_i/V_i$ implies that
\begin{equation}
\label{eq: injection inequality}
\dim_{\fp}(W_i) - \dim_{\fp}(W_i \cap V_i ) \le \dim_{\fp}(U_i) - \dim_{\fp}(V_i) = d(i) - \dim_{\fp}(V_i).
\end{equation}
As in the proof of Theorem \ref{thm: stair main theorem}, we have
$$
\bP(\det(\ol{X_n}) \ne 0) = \prod_{i=1}^{n} \bP(x_i \notin V_i \mid \dim_{\fp}(V_i) = i-1 ).
$$
Since $x_i \in W_i$ for any $1\le i \le n$, we have
\begin{align*}
\bP(x_i \notin V_i \mid \dim_{\fp}(V_i) = i-1 ) & = \bP(x_i \in W_i\setminus (W_i \cap V_i) \mid \dim_{\fp}(V_i) = i-1 ) \\
& \le \left(1- \frac{p^{i-1}}{p^{d(i)}}\right),
\end{align*}
where the last inequality holds by \eqref{eq: injection inequality}. Hence
$$
\bP(\det(\ol{X_n}) \ne 0) \le \prod_{i=1}^n \left(1- \frac{p^{i-1}}{p^{d(i)}}\right) = \left(1-\frac{1}{p^{t_n+1}}\right)^{n-t_n}\prod_{j=1}^{t_n} \left(1-\frac{1}{p^j}\right).
$$
If $\underset{n \to \infty}{\lim} \bP(\det(\ol{X_n}) \ne 0) = c_p$, then 
$$
c_p \le \liminf_{n \to \infty}\left(1-\frac{1}{p^{t_n+1}}\right)^{n-t_n}\prod_{j=1}^{t_n} \left(1-\frac{1}{p^j}\right).
$$
Since $c_p> 0$, the expression on the right-hand side is bounded away from zero for large $n$.  If $t_n$ does not diverge to $\infty$, then along a subsequence where $t_n$ is bounded, we have $\left(1-\frac{1}{p^{t_n+1}}\right)^{n-t_n} \rightarrow 0$, which is a contradiction. 
Hence
$\prod_{j=1}^{t_n}(1-p^{-j})$ converges to $c_p$, and the above inequality forces
$$
\lim_{n\to\infty}\left(1-\frac{1}{p^{t_n+1}}\right)^{n-t_n}=1.
$$
Equivalently, $(n-t_n)/p^{t_n+1}\to 0$, and hence $t_n-\log_p n\to\infty$.
\end{proof}

\begin{rem}
In the proof of (1) $\Longrightarrow$ (2) in Theorem \ref{thm: Meszaros thm}, Mészáros introduced the notion of an $L$-localized kernel. He showed that if $t_n -\log_p n$ does not diverge to infinity as $n \to \infty$, then $\ol{X_n}$ has many localized vectors supported on disjoint blocks with high probability.
This implies that the limiting distribution of $\cok(\ol{X_n})$ has a heavier tail than the Cohen--Lenstra distribution modulo $p$; that is, there are infinitely many $n$ for which, for all sufficiently large $L$,
$$
\bP(\dim_{\fp} (\cok(\ol{X_n})) > L) \gg \sum_{m > L} \cl(m).
$$ 
The proof of Theorem \ref{thm: band main theorem} provides a simpler direct proof of the implication (1) $\Longrightarrow$ (2) of Theorem \ref{thm: Meszaros thm}. Indeed, if $t_n - \log_p n$ does not diverge to infinity as $n \to \infty$, then $\bP(\cok(\ol{X_n}) = 0)$ does not converge to $c_p$, so $\cok(X_n)$ cannot converge to CL.
\end{rem}

\subsection{Two symmetric stair-shaped zero regions} \label{Sub24}

A band matrix in Section \ref{Sub23} can be viewed as a matrix whose zero entries form two symmetric stair-shaped regions, with each step having height $1$ and width $1$. This can be generalized to the case where the zero entries form two symmetric stair-shaped regions, with each step having height $d \ge 2$ and width $1$. The case with height $1$ and width $d \ge 2$ is essentially the same, since the cokernel of a square matrix is isomorphic to that of its transpose (cf. the proof of \cite[Lemma 7.17]{KLY24}). We adopt the case with height $d \ge 2$ and width $1$, as it is more convenient for computing the moments. Let $(t_n)_{n=1}^{\infty}$ be a sequence of nonnegative integers such that $dt_n \le n$ for all $n$. Define $\Sg_n=(\sg_{n,1}, \ldots, \sg_{n,n})$ by 
\begin{equation}\label{eq: two stairs sigma}
\sg_{n, i} = \begin{cases}
[n-d(t_n-i+1)] & ~\text{if $1 \le i \le t_n$,} \\
[n] & ~\text{if $t_n + 1 \le i \le n-t_n$,} \\
[n] \setminus [d(i-n+t_n)] & ~\text{if $n- t_n + 1 \le i \le n$}.
\end{cases}
\end{equation}
Since $dt_n \le n$, the index ranges $1 \le i \le t_n$ and $n-t_n+1 \le i \le n$ do not overlap.

\begin{figure}[ht]
\begin{equation*}
\begin{pmatrix}
* & * & * & * & 0 & 0 & 0\\ 
0 & * & * & * & 0 & 0 & 0\\ 
0 & * & * & * & * & 0 & 0\\ 
0 & 0 & * & * & * & 0 & 0\\ 
0 & 0 & * & * & * & * & 0\\ 
0 & 0 & 0 & * & * & * & 0\\ 
0 & 0 & 0 & * & * & * & *
\end{pmatrix}
\end{equation*}
\caption{A matrix $X_n \in \M_n(\zp)$ for $(n, d, t_n)=(7,2,3)$}.
\label{fig4}
\end{figure}

We first recall a result from \cite{KLY24}, which will be used in the proof of Theorem \ref{thm: two stairs, height d width 1}.

\begin{lem}\label{lem: Prop 3.2 in KLY24}
(\cite[Proposition 3.2]{KLY24})
Let $\Sg_n = (\sg_{n,1}, \ldots, \sg_{n,n})$ and $\Sg_n' = (\sg_{n,1}', \ldots, \sg_{n,n}')$. Suppose that for every $n \ge 1$ and $1\le i \le n$, $\sg_{n,i} \subseteq \sg_{n,i}'$. Let $X_n$ and $X_n'$ be Haar-random matrices in $\M_n(\zp)$ supported on $\Sg_n$ and $\Sg_n'$, respectively. Then
$$
\underset{n \to \infty}{\lim} \bE(\# \Sur(\cok(X_n), G)) = 1 \quad\Longrightarrow \quad \underset{n \to \infty}{\lim} \bE(\# \Sur(\cok(X_n'), G)) = 1.
$$
\end{lem}

The proof of the following theorem relies on Proposition \ref{prop: analogue of proposition 3.2 of KLY24}, which will be proved in Section \ref{Sec4}.

\begin{thm} \label{thm: two stairs, height d width 1}
Let $X_n \in \M_n(\zp)$ be the Haar-random matrix supported on $\Sg_n = (\sg_{n,1}, \ldots, \sg_{n,n})$, where $\sg_{n,i}$ are given as in \eqref{eq: two stairs sigma}. Then the following are equivalent:
\begin{enumerate}
    \item $\underset{n \to \infty}{\lim}\bP(\det(\ol{X_n}) \ne 0)=c_p$,
    \item $\underset{n \to \infty}{\lim} (n-dt_n) = \infty$,
    \item $\cok(X_n) \rightrightarrows \CL$.
\end{enumerate}
\end{thm}

\begin{proof}
The implication (3) $\Longrightarrow$ (1) is trivial.
If condition (1) is true, then Proposition \ref{prop: analogue of proposition 3.2 of KLY24} implies that $\underset{n \to \infty}{\lim}\bP(\det(\ol{X'_n}) \ne 0)=c_p$ where $X'_n \in \M_n(\zp)$ is the Haar-random matrix supported on $\Sg'_n=(\sg'_{n,1}, \ldots, \sg'_{n,n})$ given by $\sg'_{n,i} = \sg_{n,i}$ for $1 \le i \le t_n$ and $\sg'_{n,i}=[n]$ for $t_n+1 \le i \le n$. By Corollary \ref{cor: stair d>1} (applied to the transpose of $X'_n$), we deduce that $\underset{n \to \infty}{\lim} (n-dt_n) = \infty$. (Note that $t_n$ in this theorem corresponds to $n-t_n$ in Corollary \ref{cor: stair d>1}.)

Now we prove that (2) implies $\underset{n \to \infty}{\lim}\bE(\# \Sur(\cok(X_n), G)) = 1$ for every finite abelian $p$-group $G$. Applying Theorem \ref{thm: Wood robust uniqueness theorem}, this proves (3).
By Lemma \ref{lem: Prop 3.2 in KLY24}, we may enlarge $t_n$ (if necessary) so that $n < 2dt_n$ for every $n$ while still ensuring that $\underset{n \to \infty}{\lim} (n - dt_n) = \infty$.
First, we slightly reduce the subsets $\sg_{n,i}$ to simplify the computation of the moments. 
Let $r_n$ be the remainder when $n$ is divided by $d$, $\sg''_{n,i} = \sg_{n,i}$ for $1 \le i \le n-t_n$ and $\sg''_{n,i}=[n] \setminus [d(i-n+t_n)+r_n]$ for $n-t_n+1 \le i \le n$. By Lemma \ref{lem: Prop 3.2 in KLY24}, it is enough to prove $\underset{n \to \infty}{\lim}\bE(\# \Sur(\cok(X''_n), G)) = 1$ where $X''_n \in \M_n(\zp)$ is the Haar-random matrix supported on $\Sg''_n=(\sg''_{n,1}, \ldots, \sg''_{n,n})$. This modification makes the number of supporting entries in each of the relevant rightmost columns a multiple of $d$, so that these entries can be grouped into blocks of size $d$ in the computation below.

For a finite abelian $p$-group $G$, let $F \in \Sur(\zp^n, G)$ and $FV_{\sg''_{n,i}}=G_i$ for $1 \le i \le n$. Then we have
\begin{itemize}
    \item $G_1 \le G_2 \le \cdots \le G_{t_n} \le G$, 
    \item $G_{t_n+1}=G_{t_n+2}=\cdots=G_{n-t_n}=G$, and
    \item $G \ge G_{n-t_n+1} \ge G_{n-t_n+2} \ge \cdots \ge G_n$.
\end{itemize}
Define $d''_{G_1, \ldots, G_n}$ as in Section \ref{Sub21}, except that $\sg_{n,i}$ is replaced by $\sg''_{n,i}$. For $s_n := \lf \frac{n}{d}-t_n-1 \rf$, we have $n=d(t_n+s_n+1)+r_n$ 
where the assumption $n<2dt_n$ implies $s_n < t_n$. 
Since $\underset{n \to \infty}{\lim} s_n = \infty$ by condition (2), we may assume that $n$ is sufficiently large so that $s_n \ge 1$. Then 
\begin{alignat*}{3}
   &\sg''_{n,1} \setminus \sg''_{n,n-t_n+1}   &\;=\;& [d+r_n], \\
   &\sg''_{n,1} \cap (\sg''_{n,n-t_n+k}\setminus \sg''_{n,n-t_n+k+1})
   &\;=\;& [d(k+1)+r_n]\setminus [dk+r_n]   &\quad& \text{for } 1 \le k \le s_n,\\
   &\sg''_{n,k} \cap \sg''_{n,n-(t_n-s_n+1)+k}
   &\;=\;& [d(s_n+k)+r_n]\setminus [d(s_n+k-1) + r_n]   &\quad& \text{for } 2 \le k \le t_n-s_n,\\
   &(\sg''_{n,k}\setminus \sg''_{n,k-1}) \cap \sg''_{n,n} 
   &\;=\;& [d(s_n+k)+r_n]\setminus [d(s_n+k-1) + r_n]   &\quad& \text{for } t_n-s_n+1 \le k \le t_n.
\end{alignat*}
Note that this collection of sets is pairwise disjoint.
Then it follows that
\begin{equation} \label{eq29a}
\begin{split}
\sum_{\substack{(G_1, \ldots, G_n) \\ \ne (G, \ldots, G)}} d''_{G_1, \ldots, G_n}
& = \sum_{\substack{G_1 \le G_2 \le \cdots \le G_{t_n} \\ G_n \le G_{n-1} \le \cdots \le G_{n-t_n+1} \\ (G_1, G_n) \ne (G,G)}} \frac{\# \left \{ F \in \Sur(\zp^n, G) \mid FV_{\sg''_{n,i}}=G_i \text{ for } 1 \le i \le n
 \right \}}{|G_1| \cdots |G_{t_n}| \cdot |G_{n-t_n+1}| \cdots |G_n| \cdot |G|^{n-2t_n}} \\
& \le \sum_{\substack{G_1 \le G_2 \le \cdots \le G_{t_n} \\ G_n \le G_{n-1} \le \cdots \le G_{n-t_n+1} \\ (G_1, G_n) \ne (G,G)}} \frac{|G_1|^{d+r_n} (A_1A_2A_3)^d |G_n|^d}{|G_1| \cdots |G_{t_n}| \cdot |G_{n-t_n+1}| \cdots |G_n| \cdot |G|^{n-2t_n}}
\end{split}
\end{equation}
where
\begin{align*}
A_1 &:=  \prod_{k=1}^{s_n} |G_1 \cap G_{n-t_n+k}|, \\
A_2 &:= \prod_{k=2}^{t_n-s_n} |G_k \cap G_{n-(t_n-s_n+1)+k}|, \\
A_3 &:= \prod_{k=t_n-s_n+1}^{t_n} |G_k \cap G_n|.
\end{align*}

If $1 \le i \le t_n$, $n-t_n+1 \le j \le n$ and
$j-i \le n-(t_n-s_n+1)$, then $\sg''_{n,i} \cup \sg''_{n,j} = [n]$. In this case, 
$G_i+G_j = FV_{\sg''_{n,i}}+FV_{\sg''_{n,j}} = G$ so 
$$
|G_i \cap G_j| = \frac{|G_i||G_j|}{|G_i+G_j|} = \frac{|G_i||G_j|}{|G|}.
$$
Letting $B_1 := \prod_{k=1}^{t_n} |G_k|$ and $B_2 := \prod_{k=n-t_n+1}^{n} |G_k|$, this implies that
\begin{equation} \label{eq29b}
\begin{split}
& \frac{|G_1|^{d+r_n} (A_1A_2A_3)^d |G_n|^d}{|G_1| \cdots |G_{t_n}| \cdot |G_{n-t_n+1}| \cdots |G_n| \cdot |G|^{n-2t_n}} \\
= \, & \frac{|G_1|^{d+r_n}|G_n|^d}{B_1B_2 |G|^{n-2t_n}} \lt( \frac{|G_1|^{s_n-1} B_1 B_2 |G_n|^{s_n-1}}{|G|^{t_n+s_n-1}}\rt)^d \\
= \, & \frac{|G_1|^{ds_n+r_n} B_1^{d-1}}{|G|^{ds_n+r_n+(d-1)t_n}} 
\cdot \frac{|G_n|^{ds_n} B_2^{d-1}}{|G|^{ds_n+(d-1)t_n}} \\
\le \, & \frac{|G_1|^{ds_n} B_1^{d-1}}{|G|^{ds_n+(d-1)t_n}} 
\cdot \frac{|G_n|^{ds_n} B_2^{d-1}}{|G|^{ds_n+(d-1)t_n}}.
\end{split}
\end{equation}

Let $|G|=p^m$. Suppose that
$$
G_1=\cdots=G_{i_1} \lneq G_{i_1+1}=\cdots=G_{i_2} \lneq \cdots = G_{i_r} \lneq G_{i_r+1}=\cdots=G_{t_n+1} = G
$$
for some $1 \le i_1<\cdots<i_r \le t_n$, and
$$
G_n=\cdots=G_{n-j_1+1} \lneq G_{n-j_1}=\cdots=G_{n-j_2+1} \lneq \cdots = G_{n-j_s+1} \lneq G_{n-j_s}=\cdots=G_{n-t_n} = G
$$
for some $1 \le j_1<\cdots<j_s \le t_n$. Then
\begin{equation} \label{eq29c}
\frac{|G_1|^{ds_n} B_1^{d-1}}{|G|^{ds_n+(d-1)t_n}} 
\cdot \frac{|G_n|^{ds_n} B_2^{d-1}}{|G|^{ds_n+(d-1)t_n}} \le \frac{1}{p^{rds_n+(d-1)(i_1+\cdots+i_r)}} \cdot \frac{1}{p^{sds_n+(d-1)(j_1+\cdots+j_s)}}.
\end{equation}
Since $|G|=p^m$, we have $r,s \le m$. Moreover, the condition $(G_1, G_n) \ne (G,G)$ implies that at least one of $r$ and $s$ is nonzero.
Combining \eqref{eq29a}, \eqref{eq29b} and \eqref{eq29c}, we obtain
\begin{align*}
\sum_{\substack{(G_1, \ldots, G_n) \\ \ne (G, \ldots, G)}} d''_{G_1, \ldots, G_n} 
& \le O_G(1) \sum_{\substack{r+s \ge 1 \\ r,s \le m}} \sum_{\substack{1 \le i_1 < \cdots < i_r \le t_n \\ 1 \le j_1 < \cdots < j_s \le t_n}} \frac{1}{p^{rds_n+(d-1)(i_1+\cdots+i_r)}} \frac{1}{p^{sds_n+(d-1)(j_1+\cdots+j_s)}} \\
& \le O_G(1) \sum_{\substack{r+s \ge 1 \\ r,s \le m}} \frac{1}{p^{(r+s)ds_n}} \lt( \sum_{i=1}^{\infty} \frac{1}{p^{(d-1)i}} \rt)^{r+s} \\
& \le O_G \lt( \lt( 1+ \sum_{r=1}^{m} \frac{1}{p^{rds_n}} \rt) \lt( 1+\sum_{s=1}^{m} \frac{1}{p^{sds_n}} \rt)-1 \rt).
\end{align*}
Now the condition $\underset{n \to \infty}{\lim} s_n = \infty$ implies that
$$
\lim_{n \to \infty} \sum_{\substack{(G_1, \ldots, G_n) \\ \ne (G, \ldots, G)}} d''_{G_1, \ldots, G_n} = 0.
$$
By Lemma \ref{lem: d moment error term}, we have $\underset{n \to \infty}{\lim}\bE(\# \Sur(\cok(X''_n), G)) = 1$, which completes the proof.
\end{proof}

\begin{rem}
Let $(R, \mf{m})$ be a complete discrete valuation ring with finite residue field $R/\mf{m} = \fq$, where $\fq$ denotes the finite field with $q$ elements. Let $\mathrm{Mod}_R$ be the set of isomorphism classes of finite $R$-modules. Let $(Y_n)_{n=1}^{\infty}$ be a sequence of random finitely generated $R$-modules.
Extending our notation, we write $Y_n \rightrightarrows \CL$ if for every $H \in \mathrm{Mod}_R$,
$$
\lim_{n \to \infty} \bP(Y_n \cong H) = \frac{1}{\#\Aut_R(H)}\prod_{i=1}^\infty (1- q^{-i}). 
$$
Here, $\Aut_R(H)$ denotes the set of $R$-module automorphisms of $H$. 
Then all results in this section remain valid when $\zp$ and $p$ are replaced by $R$ and $q$, respectively. To apply the moment method in this setting, we require the results of \cite{SW22}, in particular \cite[Lemma 6.9]{SW22}.
\end{rem}

\subsection{Extension to \texorpdfstring{$n \times (n+u)$}{n by (n+u)} matrices} \label{Sub25}

Let $u$ be a nonnegative integer. In this section, we extend our results to $n\times (n+u)$ rectangular matrices.

\begin{thm}[Wood, {\cite[Section 3]{Woo19}}]\label{thm: Wood robust uniqueness theorem2}
Let $\mc{P}$ be the set of all finite abelian $p$-groups and $Y_1, Y_2, \ldots$ be a sequence of random finitely generated $\zp$-modules. Suppose that
$$
\underset{n \to \infty}{\lim}\bE(\#\Sur(Y_n, H)) = \frac{1}{|H|^u}
$$
for all $H \in \mc{P}$. Then for every $G \in \mc{P}$,
$$
\underset{n \to \infty}{\lim} \bP(Y_n \cong G) = \frac{1}{\#\Aut(G)|G|^u}\prod_{i=1}^\infty (1-p^{-u-i}).
$$
\end{thm}

We define the collection of support patterns $\Sg_n = (\sg_{n,1}, \ldots, \sg_{n,n+u})$ in analogy with our definitions from Sections \ref{Sub22}, \ref{Sub23}, and \ref{Sub24}. More precisely, for each $i \in [n+u]$, we define $\sg_{n,i}$ according to one of the following cases:
\begin{enumerate}[(i)]
\item
\begin{equation*}
\sg_{n, i} = \begin{cases}
[v(n)_a] & \text{if } h(n)_{a-1} + 1 \le i \le h(n)_a \text{ for some } 1 \le a \le \ell_n, \\
[n] & \text{if } h(n)_{\ell_n} + 1 \le i \le n+u,
\end{cases}
\end{equation*}

\item
\begin{equation*}
\sg_{n,i} = \{j \in [n] : i-t_n-u \le j \le i+t_n\},
\end{equation*}

\item
\begin{equation*}
\sg_{n, i} = \begin{cases}
[n-d(t_n-i+1)] & \text{if } 1 \le i \le t_n, \\
[n] & \text{if } t_n + 1 \le i \le n+u-t_n, \\
[n] \setminus [d(i-n-u+t_n)] & \text{if } n+u- t_n + 1 \le i \le n+u.
\end{cases}
\end{equation*}
\end{enumerate}
Using Theorem \ref{thm: Wood robust uniqueness theorem2}, the arguments in the proofs of Theorems \ref{thm: stair main theorem}, \ref{thm: band main theorem}, and \ref{thm: two stairs, height d width 1} extend in a similar manner to yield the following result.
We omit the details.

\begin{thm}
Let $X_n \in \M_{n\times (n+u)}(\zp)$ be the Haar-random matrix supported on $\Sg_n = (\sg_{n,1}, \ldots, \sg_{n,n+u})$, where the supports $\sg_{n,i}$ are given by one of (i)--(iii) above. Then the following are equivalent. 
\begin{enumerate}
\item
For every finite abelian $p$-group $G$,
$$
\underset{n \to \infty}{\lim} \bP(\cok(X_n) \cong G) = \frac{1}{\#\Aut(G)|G|^u}\prod_{i=1}^\infty (1-p^{-u-i}).
$$
\item 
$$
\underset{n \to \infty}{\lim} \bP(\ol{X_n} \text{ has full rank $n$}) = \prod_{i=1}^\infty (1- p^{-u-i}).
$$
\end{enumerate}
\end{thm}

\section{Rank-\texorpdfstring{$r$}{r} convergence does not imply CL convergence} \label{Sec3}

In this section, we show that Conjecture \ref{conj: main conjecture} fails if the assumption of convergence to CL in rank $0$ is replaced by convergence to CL in rank $r$ for some positive integer $r$. More precisely, we construct a sequence of Haar-random matrices $X_n \in \M_n(\zp)$ supported on $\Sg_n$ for which $\cok(X_n)$ converges to CL in rank $r$, but does not converge to CL in rank $0$. 

\begin{lem} \label{lem_product}
For every real number $\alpha \in (0,1)$, there exists a sequence of positive integers $(j_k)_{k=1}^{\infty}$ such that 
$$
\prod_{k=1}^{\infty} \lt( 1 - \frac{1}{p^{j_k}} \rt) = \alpha.
$$
\end{lem}

\begin{proof}
Since $\alpha < 1$, there exists a unique positive integer $j_1$ such that 
$$
1-\frac{1}{p^{j_1-1}} \le \alpha < 1-\frac{1}{p^{j_1}}.
$$
Moreover, since 
$$
\frac{\alpha}{1-\frac{1}{p^{j_1}}}<1,
$$
there exists a unique positive integer $j_2 \ge j_1$ such that
$$
1-\frac{1}{p^{j_2-1}} \le \frac{\alpha}{1-\frac{1}{p^{j_1}}} < 1-\frac{1}{p^{j_2}}.
$$
In general, let 
$$
\alpha_n := \frac{\alpha}{\prod_{k=1}^n \lt(1-\frac{1}{p^{j_k}}\rt)} < 1
$$
and $j_{n+1} 
\ge j_n$ be the unique positive integer satisfying
$$
1-\frac{1}{p^{j_{n+1}-1}} \le \alpha_n < 1-\frac{1}{p^{j_{n+1}}}.
$$
By the inequality
$$
\alpha < \prod_{k=1}^{n} \lt(1-\frac{1}{p^{j_k}}\rt) \le \exp\lt(-\sum_{k=1}^n p^{-j_k}\rt),
$$
we have $\underset{n \to \infty}{\lim} j_n = \infty$. Now the inequality
$$
1-\frac{1}{p^{j_{n+1}-1}} \le \alpha_n < 1-\frac{1}{p^{j_{n+1}}}
$$
implies that $\underset{n \to \infty}{\lim} \alpha_n = 1$, and hence $\prod_{k=1}^{\infty} \lt( 1 - \frac{1}{p^{j_k}} \rt) = \alpha$.
\end{proof}

\begin{prop} \label{prop_converse not true finite field}
Let $r$ be a positive integer. Then there exists a sequence of support sets $(\Sg_n)_{n=1}^{\infty}$ such that if $\mathfrak{X}_n$ is a uniformly random matrix in $\M_n(\fp)$ supported on $\Sg_n$, we have
$$
\lim_{n \to \infty} \mathbb{P}(n - \mathrm{rank}(\mathfrak{X}_n) = r) = \cl(r),
$$
but for all poistive integers $m$
$$
\mathbb{P}(m-\mathrm{rank}(\mathfrak{X}_m) = 0) = 0.
$$    
\end{prop}

\begin{proof}
First suppose that $(p,r)=(2,1)$. Define $\sg_{n,1}=\varnothing$ and $\sg_{n,2}=\cdots = \sg_{n,n}=[n]$. Then 
$$
\bP(\mathrm{rank}(\mathfrak{X}_n)=n-1) = \prod_{i=2}^{n}(1-2^{-i})
$$
so
$$
\underset{n \to \infty}{\lim}\bP(\mathrm{rank}(\mathfrak{X}_n)=n-1) = 2c_2=\cl(1).
$$
On the other hand, $\bP(\mathrm{rank}(\mathfrak{X}_m)=m)=0$ for all $m$.

Hence we may assume that $(p,r)\ne (2,1)$. Then
$p^{r^2}c_p(r)>1$, so Lemma \ref{lem_product} implies that there exists a sequence of positive integers $(j_k)_{k=1}^{\infty}$ such that
$$
\prod_{k=1}^{\infty} \lt( 1 - \frac{1}{p^{j_k}} \rt) = \frac{1}{p^{r^2}c_p(r)}.
$$
Fix $n \ge 2r$ and let $t_n$ be the largest nonnegative integer such that $\sum_{k=1}^{t_n} j_k \le \frac{n}{2}$. Define 
\[
\sg_{n,i}  = \begin{cases}
       \varnothing  & \text{ for } 1\le i \le r,\\
     \{ x \in [n] \mid j_1+\cdots+j_{k-1} < x \le j_1+\cdots+j_k \}  &  \text{ for } i=r+k, \, 1 \le k \le t_n,\\
     [n] &  \text{ for } r+t_n < i \le n.
    \end{cases}
\]

Let $w_i \in \fp^n$ be the $i$-th column of $\mathfrak{X}_n$. Then 
\begin{align*}
\bP(\mathrm{rank}(\mathfrak{X}_n)=n-r)
&= \bP(w_{r+1}, w_{r+2}, \ldots, w_n \text{ are linearly independent}) \\
& = \bP(w_{r+1}, \ldots, w_{r+t_n} \ne 0) \\
& \times \bP(w_i \notin \langle w_{r+1}, \ldots, w_{i-1} \rangle ~\text{for all $r+t_n+1 \le i \le n$} \mid w_{r+1}, \ldots, w_{r+t_n} \ne 0) \\
& = \prod_{k=1}^{t_n} \lt( 1 - \frac{1}{p^{j_k}} \rt) \prod_{i=r+t_n+1}^{n} \lt(1 - \frac{1}{p^{n-(i-r-1)}} \rt) \\
& = \prod_{k=1}^{t_n} \lt( 1 - \frac{1}{p^{j_k}} \rt) \frac{c_p(n-t_n)}{c_p(r)}.
\end{align*}
Since $\underset{n \to \infty}{\lim} t_n = \infty$ and $t_n \le \sum_{k=1}^{t_n} j_k \le \frac{n}{2}$, we have
$$
\underset{n \to \infty}{\lim} \bP(\mathrm{rank}(\mathfrak{X}_n)=n-r)= \prod_{k=1}^{\infty} \lt( 1 - \frac{1}{p^{j_k}} \rt) \frac{c_p}{c_p(r)}
= \frac{1}{p^{r^2}}\frac{c_p}{c_p(r)^2} = \cl(r). 
$$
It is clear that $\bP(\mathrm{rank}(\mathfrak{X}_m)=m)=0$ for all $m$. 
\end{proof}

Let $X_n$ be the Haar-random matrix in $\M_n(\zp)$ supported on $\Sg_n$, where $\Sg_n$ is as in Proposition \ref{prop_converse not true finite field}. Then we obtain the following corollary.

\begin{cor} \label{cor_converse not true}
Let $r$ be a positive integer. Then there exists a sequence $(\Sg_n)_{n=1}^{\infty}$ such that if $X_n$ is the Haar-random matrix in $\M_n(\zp)$ supported on $\Sg_n$, then $\cok(X_n)$ converges to the Cohen--Lenstra distribution in rank $r$, but does not converge to the Cohen--Lenstra distribution in rank $0$.
\end{cor}

\section{Nonsingularity probability growth under support expansion}
\label{Sec4}
For a positive integer $n$, let $\Sg_n = (\sg_{n,1}, \ldots, \sg_{n,n})$ and $\Sg_n' = (\sg_{n,1}', \ldots, \sg_{n,n}')$. We write $\Sg_n \le \Sg_n'$ if $\sg_{n,i} \subseteq \sg_{n,i}'$ for each $i \in [n]$. We define $|\Sg_n| = |\sg_{n,1}| + \cdots + |\sg_{n,n}|$, and similarly for $|\Sg_n'|$.

\begin{lem}
\label{lem: determinant nonzero inequality}
Suppose that $\Sg_n \le \Sg_n'$. Let $\ol{X_n}$ (resp. $\ol{X'_n}$) be a uniform random matrix in $\M_n(\fp)$ supported on $\Sg_n$ (resp. $\Sg_n'$). Then
$$
\bP(\det(\ol{X_n}) \ne 0 ) \le \bP(\det(\ol{X'_n}) \ne 0).
$$
\end{lem}

\begin{proof}
It suffices to prove the lemma when $|\Sg_n'| = |\Sg_n| +1$. Furthermore, after permuting rows and columns, we may assume that $n \notin \sg_{n,1}$, $\sg'_{n,1} = \sg_{n,1} \cup \{n\}$ and $\sg'_{n,i} = \sg_{n,i}$ for all $2 \le i \le n$. In other words, the two support patterns differ only at the $(n,1)$-entry, which is fixed to be zero in $\ol{X_n}$ and is random in $\ol{X'_n}$.
Let $c_i$ (resp. $c_i')$ be the $i$-th column of $\ol{X_n}$ (resp. $\ol{X'_n}$),
$$
\pi : \fp^n \to \fp^{n-1}
$$ 
be the projection to the first $n-1$ coordinates and $\wt{c_i} = \pi(c_i) = \pi(c_i')$. Define
$$
T := \left\{t= (t_1, \ldots, t_n) \in (\fp^{n-1})^n \mid t_1,\ldots, t_n ~\text{span $\fp^{n-1}$}\right\}.
$$
Since there is a $1$-dimensional subspace of dependencies for each $t = (t_1,\ldots, t_n) \in T$, we may fix a choice of 
\[
a(t) := (a_1(t), \ldots, a_n(t)) \in \fp^n \setminus \{ (0, \ldots, 0) \}
\]
for which $\sum_{i=1}^n a_i(t)t_i = 0$. 
Let $(l_1, \ldots, l_n)$ (resp. $(l_1', \ldots, l_n')$) be the last row of the matrix $\ol{X_n}$ (resp. $\ol{X'_n}$). 
In particular, we may construct the matrices such that $l_i = l_i'$ for all $2 \le i \le n$, while $l_1=0$ and $l_1'$ is a uniform random element in $\fp$.

In this setting, we have
\begin{align*}
\bP(\det(\ol{X_n}) \ne 0 ) & = \bP(\text{$c_1, \ldots, c_n$ are linearly independent}) \\
& = \sum_{t = (t_1, \ldots, t_n) \in T} 
\bP((\wt{c_1}, \ldots, \wt{c_n}) = t) ~
\bP\lt( \sum_{i=1}^n a_i(t)l_i \ne 0 \rt)
\end{align*}
and
\begin{align*}
\bP(\det(\ol{X'_n}) \ne 0 ) & = \bP(\text{$c_1', \ldots, c_n'$ are linearly independent}) \\
& = \sum_{t = (t_1, \ldots, t_n) \in T} \bP\lt(\lt(\wt{c_1}, \ldots, \wt{c_n}\rt) = t\rt) ~ \bP\lt( \sum_{i=1}^n a_i(t)l_i' \ne 0\rt) \\
& = \sum_{t = (t_1, \ldots, t_n) \in T} \bP\lt(\lt(\wt{c_1}, \ldots, \wt{c_n}\rt) = t\rt) ~ \sum_{j=0}^{p-1} \frac{1}{p} \bP\lt( a_1(t)j+\sum_{i=2}^n a_i(t)l_i \ne 0 \rt). 
\end{align*}

Now it suffices to show that for each $t \in T$ and $1\le j \le p-1$,
$$
\bP\lt( \sum_{i=2}^n a_i(t)l_i \ne 0\rt) \le \bP\lt(a_1(t)j + \sum_{i=2}^n a_i(t)l_i \ne 0\rt),
$$
which is equivalent to
$$
\bP\lt(a_1(t)j + \sum_{i=2}^n a_i(t)l_i = 0\rt) \le \bP\lt( \sum_{i=2}^n a_i(t)l_i = 0\rt).
$$
Since $l_2,\ldots,l_n$ are either uniform on $\fp$ or identically zero, both sides are equal to $\frac{1}{p}$ whenever there exists an integer $2\le i\le n$ such that $l_i$ is uniform on $\fp$ and $a_i(t) \ne 0$. Otherwise, the sum $\sum_{i=2}^n a_i(t)l_i$ is identically zero so the right-hand side equals $1$. Hence the inequality holds.
\end{proof}

\begin{prop}
\label{prop: analogue of proposition 3.2 of KLY24}
Suppose that $\Sg_n \le \Sg_n'$ for all large enough $n$. Let $X_n$ (resp. $X'_n$) be the Haar-random matrix supported on $\Sg_n$ (resp. $\Sg_n'$). If $\underset{n \to \infty}{\lim} \bP(\det(\ol{X_n}) \ne 0 )=c_p$, then $\underset{n \to \infty}{\lim} \bP(\det(\ol{X'_n}) \ne 0 )=c_p$.
\end{prop}

\begin{proof}
Let $X_n''$ be the Haar-random matrix in $\M_n(\zp)$ (equivalently, the Haar-random matrix supported on $([n], \ldots, [n])$). 
By Lemma \ref{lem: determinant nonzero inequality}, for all sufficiently large $n$ we have
$$
\bP(\det(\ol{X_n}) \ne 0 ) \le \bP(\det(\ol{X'_n}) \ne 0 ) \le \bP(\det(\ol{X_n''}) \ne 0) = c_p(n).
$$
Now the proposition follows immediately. 
\end{proof}

\begin{rem}
Let $X_n$ be the Haar-random matrix in $\M_n(\zp)$ supported on $\Sg_n$. Proposition \ref{prop: analogue of proposition 3.2 of KLY24} states that if $\cok(X_n)$ converges to CL in rank $0$, then the same holds for $\cok(X_n')$ whenever $X_n'$ has larger support than $X_n$. 
In all examples in Section \ref{Sec2} where Conjecture \ref{conj: main conjecture} is verified, we in fact proved a stronger statement: namely, that convergence of $\cok(X_n)$ to CL in rank $0$ already implies $\underset{n \to \infty}{\lim}\bE(\#\Sur(\cok(X_n), G)) = 1$ (and hence the convergence to CL).
Moreover, Lemma \ref{lem: Prop 3.2 in KLY24} states that if the moments $\bE(\#\Sur(\cok(X_n), G))$ converge to $1$, then the same holds for $\cok(X_n')$ whenever $X_n'$ has larger support than $X_n$. This analogy can be regarded as supporting evidence for Conjecture \ref{conj: main conjecture}.
\end{rem}

\section*{Acknowledgments}
\hspace{3mm} 
Hyungmin Jang was supported by the National Research Foundation of Korea (NRF) grant funded by the Korea government (MSIT) (No. RS-2025-00563144). 
Nathan Kaplan was supported by NSF Grant DMS 2154223. 
Jungin Lee was supported by the National Research Foundation of Korea (NRF) grant funded by the Korea government (MSIT) (No. RS-2024-00334558 and No. RS-2025-02262988). 
Myungjun Yu was supported by the National Research Foundation of Korea (NRF) grant funded by the Korea government (MSIT) (No. RS-2025-00563144 and No. RS-2025-23525445) and by Yonsei University Research Fund (2024-22-0146).

\end{document}